\documentclass[12pt]{amsart}
\usepackage{amscd,amssymb}
\usepackage{graphicx,enumerate}

\usepackage{color}
\usepackage[arrow,matrix,arrow,ps,color,line,curve,frame]{xy}
\usepackage{pgf, tikz}
\usepackage{url}
\usepackage{enumerate}

\usepackage[colorinlistoftodos,bordercolor=red,backgroundcolor=red!20,linecolor=red,textsize=scriptsize]{todonotes}

\let\oldlabel=\label

\def\prellabel{\marginparsep=1em
    \def\label##1{\oldlabel{##1}\ifmmode\else\ifinner\else
         \marginpar{{\footnotesize\ \\ \tt
                    ##1}}\fi\fi}}
                    
\def\B{\operatorname{B}}

\def\conv{\operatorname{conv}}

\def\vertex{\operatorname{vert}}
\def\inte{\operatorname{int}}
\def\Hilb{\operatorname{Hilb}}

\def\E{\mathcal{E}}

\def\e{\mathbf{e}}

\def\np{{\operatorname{NPol}}}

\def\U{{\operatorname{U}}}

\def\gp{\operatorname{gp}}

\def\RR{\mathbb R}

\def\ZZ{\mathbb Z}
\def\NN{\mathbb N}

\let\phi=\varphi
\let\epsilon=\varepsilon

\advance\headheight1.15pt \marginparwidth=2.5cm

\newtheorem{lemma}{Lemma}[section]
 
\newtheorem{theorem}[lemma]{Theorem}

\theoremstyle{definition} 
 
\newtheorem{remark}[lemma]{Remark}

\begin{document}

\title[Normal polytopes and ellipsoids]{Normal polytopes and ellipsoids}

\author{Joseph Gubeladze}

\address{Department of Mathematics\\
         San Francisco State University\\
         1600 Holloway Ave.\\
         San Francisco, CA 94132, USA}
\email{soso@sfsu.edu}




\begin{abstract}
We show that: (1) unimodular simplices in a lattice 3-polytope cover a neighborhood of the boundary of the polytope if and only if the polytope is very ample, (2) the convex hull of lattice points in every ellipsoid in $\RR^3$ has a unimodular cover, and (3) for every $d\ge5$, there are ellipsoids in $\RR^d$, such that the convex hulls of the lattice points in these ellipsoids are not even normal. Part (c) answers a question of Bruns, Micha{\l}ek, and the author.
\end{abstract}

\maketitle

\section{Introduction}\label{INTRO}

\subsection{Main result} A convex polytope $P\subset\RR^d$ is \emph{normal} if it is \emph{lattice}, i.e., has vertices in $\ZZ^d$, and satisfies the condition
$$
\forall c\in\NN\quad\forall x\in(cP)\cap\ZZ^d\quad\exists x_1,\ldots,x_c\in P\cap\ZZ^d\qquad x_1+\cdots+x_c=x.
$$

A necessary condition for $P$ to be normal is that the subgroup 
$$
\gp(P):=\sum_{x,y\in P\cap\ZZ^d}\ZZ(x-y)\subset\ZZ^d
$$
must be a direct summand. Also, a face of a normal polytope is normal.

Normality is a central notion in toric geometry and combinatorial commutative algebra \cite{Kripo}. A weaker condition for lattice polytopes is \emph{very ample}; see Section \ref{preliminaries} for the definition. Normal polytopes define projectively normal embeddings of toric varieties whereas very ample polytopes correspond to normal projective varieties \cite[Proposition 2.1]{Ample}.

A sufficient condition for a lattice polytope $P$ to be normal is the existence of a \emph{unimodular cover,} which means that $P$ is a union of unimodular simplices. Unimodular covers play an important role in integer programming through their connection to the \emph{Integral {C}arath\'{e}odory Property}  \cite{NON-ICP,ICP,Schrijver}. 

There exist normal polytopes in dimensions $\ge5$ without unimodular cover \cite{Unico}. It is believed that all normal 3-polytopes have unimodular cover. But progress in this direction is scarce. Recent works \cite{3Smooth,3Covers} show that all lattice 3-dimensional parallelepipeds and centrally symmetric 3-polytopes with unimodular corners have unimodular cover.

The normality of the convex hull of lattice points in an ellipsoid naturally comes up in \cite{Jumps}. We consider general ellipsoids, neither centered at $0$ nor aligned with the coordinate axes. According to \cite[Theorem 6.5(c)]{Jumps}, the convex hull of the lattice points in any ellipsoid $E\subset\RR^3$ is normal. \cite[Question 7.2(b)]{Jumps} asks  whether this result extends to higher dimensional ellipsoids. 

Here we prove the following

\medskip\noindent{\bf Theorem.}\label{Main} \emph{Let $P\subset\RR^3$ be a lattice polytope, $E\subset\RR^d$ an ellipsoid, and $P(E)$ the convex hull of the lattice points in $E$.
\begin{enumerate}[\rm(a)]
\item The unimodular simplices in $P$ cover a neighborhood of the boundary $\partial P$ in $P$ if and only if $P$ is very ample.
\item If $d=3$ then the polytope $P(E)$ is covered by unimodular simplices.
\item For every $d\ge6$, there exists $E$ such that  $\gp(P(E))=\ZZ^d$ and $P(E)$ is not normal.
\end{enumerate}}

If in (c) we drop the condition $\gp(P(E))=\ZZ^d$, then ellipsoids $E\subset \RR^d$ with $P(E)$ non-normal already exist for $d=5$; see Remark \ref{Baranov}.

\subsection{Preliminaries}\label{preliminaries} $\ZZ_+$ and $\RR_+$ denote the sets of non-negative integers and reals, respectively.

The convex hull of a set $X\subset\RR^d$ is denoted by $\conv(X)$. The relative interior of a convex set $X\subset\RR^d$ is denoted by $\inte X$. The boundary of $X$ is denoted by $\partial X=X\setminus\inte X$.

Polytopes are assumed to be convex. For a polytope $P\subset\RR^d$, its vertex set is denoted by $\vertex(P)$.

A lattice $n$-simplex $\Delta=\conv(x_0,\ldots,x_n)\subset\RR^d$ is \emph{unimodular} if $\{x_1-x_0,\ldots,x_n-x_0\}$ is a part of a basis of $\ZZ^d$.

A \emph{unimodular pyramid} over a lattice polytope $Q$ is a lattice polytope $P=\conv(v,Q)$, where the point $v$ is not in the affine hull of $Q$ and the lattice height of $v$ above $Q$ inside the affine hull of $P$ equals 1.

Cones $C$ are assumed to be \emph{pointed, rational,} and \emph{finitely generated,} which means $C=\RR_+x_1+\cdots+\RR_+x_k$, where $x_1,\ldots,x_k\in\ZZ^d$ and $C$ does not contain a nonzero linear subspace. For a cone $C\subset\RR^d$, the smallest generating set of the additive submonoid $C\cap\ZZ^d\subset\ZZ^d$ consists of the indecomposable elements of this monoid. This is a finite set, called the \emph{Hilbert basis} of $C$ and denoted by $\Hilb(C)$. See \cite[Chapter 2]{Kripo} for a detailed discussion on Hilbert bases. For a  lattice polytope $P\subset\RR^d$, we have the inclusion of finite subsets of $\ZZ^{d+1}$:
\begin{align*}
\big(P\cap\ZZ^d,1\big)\subset\Hilb(\RR_+(P,1)).
\end{align*}
This inclusion is an equality if and only if $P$ is normal. 

A lattice polytope $P$ is \emph{very ample} if $\Hilb(\RR_+(P-v))\subset P-v$ for every vertex $v\in\vertex(P)$. All normal polytopes are very ample, but already in dimension $3$ there are very ample non-normal polytopes \cite[Exercise 2.24]{Kripo}. For a detailed analysis of the discrepancy between the two properties see \cite{Ample}. 

For a cone $C\subset\RR^d$, we say that $C$ has a \emph{unimodular Hilbert triangulation (cover)} if $C$ can be triangulated (resp., covered) by cones of the form $\RR_+x_1+\cdots+\RR x_n$, where $\{x_1,\ldots,x_n\}$ is a part of a basis of $\ZZ^d$ as well as of $\Hilb(C)$.

An \emph{ellipsoid} $E\subset\RR^d$ is a set of the form
\begin{align*}
\big\{x\in\RR^d\ |\ (l_1(x)-a_1)^2+\cdots+(l_d(x)-a_d)^2=1\big\}\subset\RR^d,
\end{align*}
where $l_1,\ldots,l_d$ is a full-rank system of real linear forms and $a_1,\ldots,a_d\in\RR^d$. 


For a lattice polytope $P$, the union of unimodular simplices in $P$ will be denoted by $\U(P)$.

\section{Unimodular covers close to the boundary}\label{Dim_3} 

The following result of  Seb\H{o} was later rediscovered in \cite{Aguzzoli,Bouvier} in  a refined form in the context of toric varieties.

\begin{theorem}(\cite{Sebo})\label{Sebo-thm}
Every 3-dimensional cone $C$ has a unimodular Hilbert triangulation.
\end{theorem}

\noindent\emph{Notice.} There exist 4-dimensional cones without unimodular Hilbert triangulation \cite{Bouvier} and it is not known whether all 4- and 5-dimensional cones have unimodular Hilbert cover. According to \cite{Unico}, in all dimensions $\ge6$ there are cones without unimodular Hilbert cover.

\medskip If $P\subset\RR^3$ is very ample, then by Theorem \ref{Sebo-thm}, for every $v\in\vertex(P)$, the cone $\RR_+(P-v)$ has a unimodular Hilbert triangulation:
$$
\RR_+(P-v)=\bigcup_{T(v)}C_t,
$$
where $T(v)$ is a finite index set, depending on $v$. In particular, the following unimodular simplices form a neighborhood of $v$ in $P$: 
$$
\Delta_{v,t}=\conv(\Hilb(C_t),0)+v,\quad t\in T(v).
$$
Also, lattice polygons have unimodular triangulation \cite[Corollary 2.54]{Kripo}. Therefore, the following lemma completes the proof of Theorem (a):

\begin{lemma}\label{PROP} For a lattice polytope $P$ of an arbitrary dimension, the following conditions are equivalent:
\begin{enumerate}[\rm(a)]
\item $\U(P)$ is a neighborhood of $\partial P$ within $P$;
\item $\U(P)$ is a neighborhood within $P$ of every vertex of $P$ and $\partial P\subset\U(P)$.
\end{enumerate} 
\end{lemma}

\begin{proof} The implication (a)$\Longrightarrow$(b) is obvious.


\medskip For the opposite implication, let:
\begin{enumerate}[\rm$\centerdot$]
\item $x\in\partial P$;
\item $F$ be the minimal face of $P$ containing
$x$;
\item $v\in\vertex(F)$;
\item $T_F$ be a unimodular cover of $F$ with $\dim(F)$-simplices, contained in $F$;
\item $T_v$ be a unimodular cover of a neighbourhood of $v$ in $P$;
\item $T_{v,F}$ be the sub-family of $T_v$, consisting of simplices that have a $\dim(F)$-
dimensional intersection with $F$;
\item $T_v/F$ be the collection of faces of simplices in $T_{v,F}$, opposite to $F$ (that is, from each simplex in $T_{v,F}$ remove the $\dim(F) + 1$ vertices that lie in $F$, so that one is left with a $(\dim(P)-\dim(F))$-simplex).
\end{enumerate}
Then, the collection of $\conv(T_v/F,T_F)$ covers a neighbourhood of $x$ in $P$ and consists of unimodular simplices.
\end{proof}

\section{Unimodular covers inside ellipsoids}\label{Dim3}

\subsection{Proof of Theorem (b)} The set of normal polytopes $P\subset\RR^d$ carries a poset structure, where the order is generated by the elementary relation
$$
P\le Q\ \text{if}\ P\subset Q\ \text{and}\  \#(Q\cap\ZZ^d)=\#(P\cap\ZZ^d)+1.
$$ In \cite{Jumps} this poset is denoted by $\np(d)$. The trivial minimal elements of $\np(d)$ are the singletons from $\ZZ^d$. It is known that $\np(d)$ has nontrivial minimal elements for $d\ge4$ \cite[Exercise 2.27]{Kripo} and the first maximal elements for $d=4,5$ were found in \cite{Jumps}. It is possible that $\np(d)$ has isolated elements for some $d$.

Computer searches so far have found neither maximal nor nontrivial minimal elements in $\np(3)$ \cite{Jumps}. The next lemma is yet another evidence that all normal 3-polytopes have unimodular cover.

\begin{lemma}\label{contraction}
Let $P$ be a normal 3-polytope. If $*\le P$ in $\np(3)$ for a singleton $*\in\ZZ^3$ then $P=\U(P)$.
\end{lemma}

\begin{proof}
If $Q\le P$ is an elementary relation in $\np(d)$ and $\dim Q<\dim P$ then $P$ is a unimodular pyramid over $Q$. In this case every full-dimensional unimodular simplex $\Delta\subset P$ is the unimodular pyramid  over a unimodular simplex in $Q$ and with the same apex as $P$. On the other hand, lattice segments and polygons are unimodularly triangulable. Therefore, it is enough to show that a polytope $P\in\np(3)$ has a unimodular cover if there is a 3-polytope $Q\in\np(3)$, such that $Q$ has a unimodular cover and $Q\le P$ is an elementary relation in $\np(3)$. Assume $\{v\}=\vertex(P)\setminus Q$. By Theorem (a) we have the inclusion $P\setminus\U(P)\subset Q$. Since $Q=\U(Q)$ we have $P=\U(P)$.
\end{proof}

Call a subset $\E\subset\ZZ^d$ \emph{ellipsoidal} and a point $v\in\E$ \emph{extremal} if there is an ellipsoid $E\subset\RR^d$, such that $\E=\conv(E)\cap\ZZ^d$  and $v\in E$.

\begin{lemma}\label{cutoff}
Let $\E\subset\RR^d$ be an ellipsoidal set. Then $\E$ has an extremal point and $\E\setminus\{v\}$ is also ellipsoidal for every extremal point $v\in\E$. 
\end{lemma}

\begin{proof}
Let $\E=\conv(E)\cap\ZZ^d$ for an ellipsoid $E\subset\RR^d$. Applying an appropriate homothetic contraction, centered at the center of $E$, we can always achieve $\E\cap E\not=\emptyset$. In particular, $\mathcal E$ has an extremal point. For $v\in\E\cap E$, after changing $E$ to its homothetic image with factor $(1+\epsilon)$ and centered at $v$, where $\epsilon$ is a sufficiently small positive real number, we can further assume $\E\cap E=\{v\}$. Finally, applying a parallel translation to $E$ by $\delta(z-v)$, where $z$ is the center of $E$ and $\delta>0$ is a sufficiently small real number, we achieve $\conv(E)\cap\ZZ^d=\E\setminus\{v\}$.
\end{proof}

Next we complete the proof of Theorem (b). It follows from Lemma 3.2 that, for any natural number $d$ and an ellipsoidal set $\E\subset\ZZ^d$, there is a descending sequence of ellipsoidal sets of the form
\begin{align*}
\E=\E_k\supset \E_{k-1}\supset\ldots\supset\E_1,\ \text{with}\ \#\E_i=i\ \text{for}\ i=1,\ldots,k.
\end{align*}

By \cite[Theorem 6.5(c)]{Jumps}, for $d=3$, the $\conv(\E_i)$ are normal polytopes. Therefore, $*\le\conv(\E)$ in $\np(3)$ for some $*\in\ZZ^3$. Thus Lemma \ref{contraction} applies.\qed

\subsection{Alternative algorithmic proof in symmetric case} For the ellipsoids $E$ with center in $\frac12\ZZ^3$, there is a different proof of Theorem (b). It yields a simple algorithm for constructing a unimodular cover of $P(E)$. 

Instead of Theorem \ref{Sebo-thm} and \cite[Theorem 6.5]{Jumps} this approach uses Johnson's 1916 \emph{Circle Theorem} \cite{Circles1,Circles}. We only need Johnson's theorem to derive the following fact, which does not extend to higher dimensions: for any lattice $\Lambda\subset\RR^2$ and any ellipse $E'\subset\RR^2$, such that $\conv(E')$ contains a triangle with vertices in $\Lambda$, every parallel translate $\conv(E')+v$, where $v\in\RR^2$, meets $\Lambda$. 

Assume an ellipsoid $E\subset\RR^3$ has center in $\frac12\ZZ^3$ and $\dim(P(E))=3$ (notation as in the theorem). Assume $\U(P(E))\subsetneqq P(E)$. Because $\partial P(E)$ is triangulated by unimodular triangles, there is a unimodular triangle $T\subset P(E)$, not necessarily in $\partial P(E)$, and a point $x\in\inte T$, such that the points in $[0,x]$, sufficiently close to $x$, are not in $\U(P(E))$. For the plane, parallel to $T$ on lattice height 1 above $T$ and on the same side as $0$, the intersection $E'=\conv(E)\cap H$ is at least as large as the intersection of $\conv(E)$ with the affine hull of $T$: a consequence of the fact that $P(E)\cap\ZZ^3$ is symmetric relative to the center of $E$.  The mentioned consequence of Johnson's theorem implies that $\conv(E')$ contains a point $z\in\ZZ^3$. In particular, all points in $[x,0]$, sufficiently close to $x$ are in the unimodular simplex $\conv(T,z)\subset P(E)$, a contradiction.

\section{High dimensional ellipsoids}\label{High} 

For a lattice $\Lambda\subset\RR^d$, define a \emph{$\Lambda$-polytope} as a polytope  $P\subset\RR^d$ with $\vertex(P)\subset\Lambda$. Using $\Lambda$ as the lattice of reference instead of $\ZZ^d$,  one similarly defines \emph{$\Lambda$-normal polytopes} and \emph{$\Lambda$-ellipsoidal sets}.

Consider the lattice $\Lambda(d)=\ZZ^d+\ZZ\left(\frac12,\ldots,\frac12\right)\subset\RR^d$.
We have $\big[\ZZ^d:\Lambda(d)\big]=2$. Consider the $\Lambda(d)$-polytope
$P(d)=\conv\big(\B(d)\cap\Lambda(d)\big)$,
where $\B(d)=\big\{(\xi_1,\ldots,\xi_d)\ \big|\ \sum_{i=1}^d\big(\xi_i-\frac12\big)^2\le\frac d4\big\}\subset\RR^d$, i.e., $\partial(\B(d))$ is the circumscribed sphere for the cube $[0,1]^d$.

Consider the $d$-dimensional $\Lambda(d)$-polytope and the $(d-1)$-dimensional $\Lambda(d)$-simplex: 
\begin{align*}
&Q(d)=\conv\big((P(d)\cap\Lambda(d))\setminus\{\e_1+\cdots+\e_d\}\big),\\
&\Delta(d-1)=\conv\left(\e_1+\cdots+\e_{i-1}+\e_{i+1}+\cdots+\e_d\ |\ i=1,\ldots,d\right),
\end{align*}
where $\e_1,\ldots,\e_d\in\RR^d$ are the standard basic vectors.

\medskip\noindent\emph{Notice.} Although $P(d)\cap\ZZ^d=\{0,1\}^d$ for all $d$, yet $[0,1]^d\subsetneqq P(d)$ for all $d\ge4$. In fact, $\big(\frac12,\ldots,\frac12\big)+k\e_i\in P(d)\cap\Lambda(d)$ for $1\le i\le d$ and $-\left\lceil\frac{\sqrt d}2\right\rceil\le k\le\left\lfloor\frac{\sqrt d}2\right\rfloor$.

\begin{lemma}\label{corner}
If $d\ge5$ then $\Delta(d-1)$ is a facet of $Q(d)$ and 
$\Delta(d-1)\cap\Lambda(d)=\vertex(\Delta(d-1))$.
\end{lemma}

\begin{proof}
Assume $x=(\xi_1,\ldots,\xi_d)\in P(d)\cap\Lambda(d)$ satisfies
$\xi_1+\cdots+\xi_d\ge d-1$. We claim that there are only two possibilities: either $x=\e_1+\cdots+\e_d$ or $x=\e_1+\cdots+\e_{i-1}+\e_{i+1}+\cdots+\e_d$ for some index $i$. Since $P(d)\cap\ZZ^d=\{0,1\}^d$, only the case $x\in\left(\frac12,\ldots,\frac12\right)+\ZZ^d$ needs to be ruled out. Assume $\xi_i=\frac12+a_i$ for some integers $a_i$, where $i=1,\ldots,d$. Then we have the inequalities
\begin{align*}
\sum_{i=1}^da_i^2\le\frac d4\quad\text{and}\quad\sum_{i=1}^d a_i\ge\frac d2-1.
\end{align*}
Since the $a_i$ are integers we have $\frac d4\ge\frac d2-1$, a contradiction because $d\ge5$.
\end{proof}

\begin{lemma}\label{non_normal}
For every even natural number $d\ge6$, there exists a point in $\big(\frac d2\cdot Q(d)\big)\cap\Lambda(d)$ which does not have a representation of the form $x_1+\cdots+x_{\frac d2}$ with $x_1,\ldots,x_{\frac d2}\in Q(d)\cap\Lambda(d)$. In particular, $Q(d)$ is not $\Lambda(d)$-normal.
\end{lemma}

\begin{proof}
Consider the baricenter $\beta(d)=\frac{d-1}d\cdot(\e_1+\cdots+\e_d)$ of $\Delta(d-1)$. The point $\frac d2\cdot\beta(d)$ is the baricenter of the dilated simplex $\frac d2\cdot\Delta(d-1)$ and, simultaneously, a point in $\Lambda(d)$. Assume $\frac d2\cdot\beta=x_1+\cdots+x_{\frac d2}$ for some $x_1,\ldots,x_{\frac d2}\in Q(d)\cap\Lambda(d)$. Lemma \ref{corner} implies $x_1,\ldots,x_{\frac d2}\in\Delta(d-1)\cap\Lambda(d)=\vertex(\Delta(d-1))$. But this is not possible because the dilated $(d-1)$-simplex $c\Delta(d-1)$ has an interior point of the form $z_1+\cdots+z_c$ with $z_1,\ldots,z_c\in\vertex(\Delta(d-1))$ only if $c\ge d$.
\end{proof}

\begin{proof}[Proof of Theorem (c)] Since $\e_1,\ldots,\e_d,(\frac12,\ldots,\frac12)\in Q(d)$ we have the equality $\gp(Q(d))=\Lambda(d)$. By Lemmas \ref{cutoff} and \ref{corner}, the set $Q(d)\cap\Lambda(d)$ is $\Lambda(d)$-ellipsoidal for $d\ge5$. By applying a linear transformation, mapping $\Lambda(d)$ isomorphically to $\ZZ^d$, Lemma \ref{non_normal} already implies Theorem (c) for $d$ even. 

One involves all dimensions $d\ge6$ by observing that (i) if $\E\subset\RR^d$ is an ellipsoidal set then $\E\times\{0,1\}\subset\RR^{d+1}$ is also ellipsoidal and (ii) the normality of $\conv(\E\times\{0,1\})$ implies that of $\conv(\E)$. While (ii) is straightforward, for (i) one applies an appropriate affine transformation to achieve $\E=\conv(S^{d-1})\cap\Lambda$, where $S^{d-1}\subset\RR^d$ is the unit sphere, and $\Lambda\subset\RR^d$ is a shifted lattice. In this case the ellipsoid
$E=\big\{(\xi_1,\ldots,\xi_d)\ \ \bigg|\ \ \frac{\xi_1^2}{a^2}+\cdots+\frac{\xi_{d-1}^2}{a^2}+\frac{\xi_d^2}{a^2}+\frac{(\xi_{d+1}-\frac12)^2}{b^2}=1\big\}\subset\RR^{d+1}$ with $b>\frac12$ and $a=\frac{2b}{\sqrt{4b^2-1}}$, is within the $(b-\frac12)$-neighborhood of the region of $\RR^{d+1}$ between the hyperplanes $(\RR^d,0)$ and $(\RR^d,1)$ and satisfies the following conditions: $E\cap(\RR^d,0)=(S^{d-1},0)$ and $E\cap(\RR^d,1)=(S^{d-1},1)$. In particular, when $\frac12<b<\frac32$ we have $\conv(E)\cap(\Lambda\times\ZZ)=\E\times\{0,1\}$.
\end{proof}

\begin{remark}\label{Baranov}
The definition of a normal polytope in the introduction is stronger than the one in \cite[Definition 2.59]{Kripo}: the former is equivalent to the notion of an \emph{integrally closed} polytope, whereas `normal' in the sense of \cite{Kripo} is equivalent to $\gp(P)$-normal. Examples of $\gp(P)$-normal polytopes, which are not normal, are lattice non-unimodular simplices, whose only lattice points are the vertices. Lemma \ref{corner} and the proof of Lemma \ref{non_normal} show that the 5-simplex $\Delta(5)$ is not $\Lambda(6)$-unimodular. Applying an appropriate affine transformation we obtain a lattice non-unimodular simplices $\Delta'\subset\RR^5$ with $\vertex(\Delta')$ ellipsoidal. Such examples in $\RR^5$ have been known sine the 1970s: a construction of Voronoi \cite{Baran} yields a lattice $\Lambda\subset\RR^5$ and a 5-simplex $\Delta\subset\RR^5$ of $\Lambda$-multiplicity 2, whose circumscribed sphere does not contain points of $\Lambda$ inside except $\vertex(\Delta)$.

We do not know whether there are ellipsoidal subsets $\E\subset\RR^5$ with $\conv(\E)$ non-normal and $\gp(\conv(\E))=\ZZ^5$. For instance, $Q(5)$ is $\Lambda(5)$-normal, as checked by \textsf{Normaliz} \cite{Normaliz}.
\end{remark}

\noindent{\bf Acknowledgment.} We thank the referees for the streamlined version of the original proof of Theorem (a), bringing \cite{Baran} to our attention, and spotting several inaccuracies.

\bibliographystyle{plain}
\bibliography{bibliography}

\end{document}